\RequirePackage{fix-cm}
\documentclass[smallextended,natbib]{svjour3}

\smartqed

\usepackage{amsmath,amssymb}
\usepackage{tikz}
\usetikzlibrary{arrows.meta}
\usepackage{algorithm}
\usepackage{algorithmic}
\usepackage{xcolor}

\definecolor{links}{RGB}{204,36,29}
\usepackage[colorlinks=true,breaklinks=true,bookmarks=true,urlcolor=links,citecolor=links,linkcolor=links,bookmarksopen=false,draft=false]{hyperref}

\def\x{\boldsymbol{x}}
\def\y{\boldsymbol{y}}
\def\a{\boldsymbol{a}}
\def\q{\boldsymbol{q}}
\def\p{\boldsymbol{p}}

\def\R{\mathbb{R}}

\def \HM {\mathrm{HM}}

\journalname{Economic Theory Bulletin}

\begin{document}

\title{A Polynomial-Time Test for Peak-Oriented Rationalizability}

\titlerunning{A Polynomial-Time Test for Peak-Oriented Rationalizability}

\author{Taotao He \and Runfa Hu}

\authorrunning{T. He and R. Hu}

\institute{
Taotao He \at
Antai College of Economics and Management, Shanghai Jiao Tong University\\
\email{hetaotao@sjtu.edu.cn}
\and
Runfa Hu \at
Antai College of Economics and Management, Shanghai Jiao Tong University
}

\maketitle

\begin{abstract}
We study the computational complexity of peak-oriented rationalizability, a survey based revealed-preference test introduced by~\cite{seror2026concaverationalizationidealpoint}. We provide a polynomial-time algorithm for testing rationalizability and recovering a utility function, establishing that peak-oriented preference elicitation is computationally tractable. In contrast, we show that computing the peak-oriented Houtman-Maks index is NP-hard. These results delineate  the precise computational boundaries of peak-oriented revealed-preference analysis.
\keywords{Revealed preference \and Peak-oriented rationalizability \and Computational complexity \and Prefix tree \and Houtman--Maks index}
\end{abstract}

\noindent\textbf{JEL Classification:} D11; C63

\section{Introduction}
Revealed preference theory studies whether observed choices can be rationalized as utility maximization behavior. In classical consumer theory,~\citet{afriat1967construction} characterizes rationalizability via the generalized axiom of revealed preference (GARP). Subsequently,~\citet{varian1982nonparametric} provides a polynomial-time algorithm for testing GARP and recovering a utility function rationalizing the observed choices. This computational tractability has established revealed preference theory as a central tool for  demand analysis. However, while researchers frequently seek to test preferences with additional properties \citep[see][Section 4]{chambers2016revealed}, such tractability does not automatically hold under additional structures; for example, testing for separable rationalization is  NP-complete~\citep{echenique2014testing,cherchye2015revealed}.

In the context of survey data, where respondents complete the same survey under different linear budget constraints,~\cite{seror2026concaverationalizationidealpoint} develops an Afriat-type characterization of concave rationalization with an unknown peak point, which we review in Section~\ref{defseror}. Specifically, he introduces the notion of peak-oriented rationalization and reduces its verification to a two-step test: first, checking whether a candidate peak satisfies a cone-membership condition that determines the orientation of each budget normal; second, verifying whether the resulting oriented data satisfies a GARP-type condition. Finally, Seror presents an empirical study illustrating how this framework can elicit political preferences using survey data from French respondents.

Seror's characterization prompts a Varian-type computational question: is peak-oriented rationalizability testable in polynomial time with respect to the number of survey questions and rounds? We resolve this affirmatively in Section~\ref{test}. Given a fixed orientation pattern, the second step reduces to a standard GARP test, solvable in polynomial time. The key bottleneck is therefore the enumeration of orientation patterns generated by candidate peak points. Seror shows that $2^n$ patterns can arise theoretically; however, we prove that only at most $2(m+1)n$ patterns are relevant, where $m$ represents the number of survey questions and $n$ the number of rounds. Moreover, leveraging a prefix tree construction~\citep{fredkin1960trie, knuth1998art} over the observed answers, we obtain an algorithm that enumerates all relevant orientation patterns in polynomial time.

The computational landscape changes when the question shifts from whether the entire dataset is rationalizable to how much of it can be rationalized. When a dataset fails the peak-oriented rationalizability test, \cite{seror2026concaverationalizationidealpoint} uses a Houtman--Maks-type index to measure the degree of irrationality, defined as the largest fraction of observation rounds that form a peak-oriented rationalizable subset. In Section~\ref{hardness}, we show that computing this peak-oriented Houtman--Maks index is NP-hard. Thus, while testing exact peak-oriented rationalizability is polynomial-time solvable, measuring the maximal extent of rationalizability is computationally intractable unless P = NP.

\textbf{Notation:} For a positive integer $n$, let $[n]:=\{1, \ldots, n\}$. For two vectors $\x,\y \in \R^n$, let $\langle \x, \y \rangle$ denote the inner product, that is, $\sum_{i \in [n]}x_iy_i$. For a concave function $u$, let $\partial u(q)$ denote its superdifferential at $q$. For a real number $x$, the sign operation $\operatorname{sign}(x)$ returns $+1$ for $x>0$, $0$ for $x=0$, and $-1$ for $x<0$. We identify the signs $+$ and $-$ with the numbers $+1$ and $-1$, respectively, whenever they are used algebraically. 

\section{Peak-Oriented Rationalization} \label{defseror}
In this section, we review the peak-oriented rationalization and its characterizations proposed in~\cite{seror2026concaverationalizationidealpoint}. Consider a survey with $m$ questions. For each question $j \in [m]$, the set of possible answers is $[0, M_j]$ for some $M_j > 0$. Let  $X:=\prod_{j=1}^m[0,M_j]$ denote the space of all possible answers.  The respondent answers the survey over $n$ rounds. In each round $i\in[n]$, the respondent provides a vector of answers $\q^i$ subject to a linear budget constraint, $B^i:=\bigl\{\x\in X \bigm| \langle \mathbf{a}^i,  \x \rangle = \mu^i\bigr\}$,
where $\a^i\in\mathbb{R}^m$ is the normal vector of the budget hyperplane and $\mu^i\in\mathbb{R}$ is a constant; see~\cite{seror2026concaverationalizationidealpoint} for their interpretations in the survey setting. We maintain their assumptions that all coordinates of $\a^i$ are nonzero, $B^i$ intersects the interior of $X$, and $\q^i\in B^i$.

\begin{definition}{\citep[Definition~5]{seror2026concaverationalizationidealpoint}}
A dataset $D:=\{(\q^i,B^i)\}_{i \in [n]}$ is \emph{peak-oriented rationalizable} if there exist a peak $\y^*\in X$, a continuous concave function $u:X\to\mathbb{R}$, and scalars $\lambda^i \in\mathbb{R}$ such that
\begin{enumerate}
    \item[(i)] $\y^* \in\arg\max_{\x\in X}u(\x)$ \label{def:1};
    \item[(ii)] $\q^i\in\arg\max_{\x\in B^i}u(\x)$ for  $i\in[n]$ \label{def:2};
    \item[(iii)] $\lambda^i \a^i \in\partial u(\q^i)$ for  $i\in[n]$\label{def:3};
    \item[(iv)] $\lambda^ia^i_j (y^*_j-q^i_j) \geq 0$ for  $i\in[n]$ and $j\in[m]$ \label{def:4};
    \item[(v)] $\langle \lambda^i \a^i, \y^*- \q^i\rangle>0$ for  $i\in[n]$ with $\q^i\neq \y^*$.\label{def:5}
\end{enumerate}
\end{definition}
\noindent Conditions (i)--(ii) have the standard utility maximization interpretation. Condition (iii) is a first-order optimality condition at $\q^i$. Conditions (iv)--(v) require the selected supergradient at each observation to point coordinatewise toward the peak $\y^*$, with strict progress whenever $\q^i\neq \y^*$.
When $\q^i=\y^*$, condition (v) is vacuous and other conditions hold automatically under $\lambda^i = 0$, so an observation at the peak does not determine the orientation of its budget normal. 

The geometry of condition (iv)  is characterized by~\citep[Lemma~1]{seror2026concaverationalizationidealpoint}. For each $i\in[n]$, define the oriented orthants
\begin{align*}
  O^i_+ &:= \bigl\{\x \in X \bigm| a_j^i(x_j-q_j^i)\geq0 \, \text{ for } j \in [m] \bigr\}\\
  O^i_- &:= \bigl\{\x \in X \bigm| a_j^i(x_j-q_j^i)\leq0 \, \text{ for } j \in [m] \bigr\},
\end{align*}
and the \emph{double cone} $C^i:=O^i_+\cup O^i_-$. Note that since coordinates of $\a^i$ are assumed to be nonzero,  $O^i_+\cap O^i_-=\{\q^i\}$ for $i \in [n]$. Thus, for $\y \neq \q^i$, the orientation conditions hold if and only if $\y \in C^i$. Moreover, such $\y$ determines a unique sign $\sigma^i(\y) \in \{+, -\}$, which equals $+$ if $\y \in O^i_+$ and $-$ if $\y \in O^i_-$. 

Theorem~2 of~\cite{seror2026concaverationalizationidealpoint} provides a characterization of peak-oriented rationalizability, which is further interpreted as a two-step test~\citep[Section 2.6]{seror2026concaverationalizationidealpoint}. First, identify a candidate peak that belongs to the double cone for each relevant observation, thereby uniquely determining the orientation of each budget normal. Second, verify that after reorienting the budget normals, the retained oriented data satisfy the Generalized Axiom of Revealed Preference (GARP). Recall that a purchase dataset $E = \{(\x^t, \p^t)\}_{t \in [T]}$ satisfies GARP if there is no sequence of indices $t_1, \ldots, t_K \in [T]$ such that~\cite[Definition~1]{dziewulski2024revealed}
\begin{equation}\label{eq:GARP}
\langle \p^{t_k},  \x^{t_k} - \x^{t_{k+1}} \rangle \geq 0  \text{ for } k\in [K-1] \quad \text{ and } \quad \langle \p^{t_K}, \x^{t_K} - \x^{t_1} \rangle  > 0. \tag{\textsc{Garp}}
\end{equation}
We formally summarize this result below.\begin{theorem}\label{them:peak-garp}
The dataset $D$ is peak-oriented rationalizable if and only if there exists $\y\in X$ such that, letting
\[
I(\y):=\{i\in[n]\mid \q^i\neq\y\},
\] 
\begin{enumerate}
    \item[(a)]  $\y \in C^i$ for every $i \in I(\y)$; and \label{them:peak-garp-a}
    \item[(b)] the retained oriented data $\bigl\{\bigl(\q^i, \sigma^i(\y)\a^i \bigr)\bigr\}_{i \in I(\y)}$ satisfy~\eqref{eq:GARP}\label{them:peak-garp-b}.
\end{enumerate}
\end{theorem}

\section{A Polynomial-Time Rationalizability Test}\label{test}
For a fixed orientation pattern $\sigma(\y):=(\sigma^i(\y))_{i\in[n]}$, condition (b) in Theorem~\ref{them:peak-garp} is a standard GARP test that can be performed in polynomial time; see~\citep{varian1982nonparametric,fostel2004two} and the survey of recent developments in~\citep{smeulders2019algorithmic}. Thus, the remaining challenge is to enumerate the orientation patterns induced by feasible peak points.

Let $Q:=\{\q^1,\ldots,\q^n\}$ denote the set of survey answers. For peaks outside $Q$, the set of feasible orientation patterns is
\begin{equation*}
\Omega:=\left\{\sigma(\y)\in\{-,+\}^n \ \middle|\ \y\in X\setminus Q,\ \y\in\bigcap_{i\in[n]}C^i\right\}.
\end{equation*}
A natural approach to enumerating $\Omega$ is to test each candidate orientation pattern for membership. 
For each $\sigma\in\{-,+\}^n$, define its associated region:
\begin{equation}\label{eq:orientation_region}
    \mathcal{R}_\sigma:=\bigcap_{i \in [n]} O^i_{ \sigma^i}
\end{equation}
The region $\mathcal{R}_\sigma$ contains precisely the points in $X$ that induce the orientation pattern $\sigma$. Consequently, membership in $\Omega$ can be characterized via the nonemptiness check:
\begin{equation}\label{eq:mem-test}
    \sigma \in \Omega \quad \text{if and only if} \quad \mathcal{R}_\sigma\setminus Q\neq\varnothing.
\end{equation}
However, exhaustive enumeration over $\{-, +\}^n$ requires $2^n$ membership tests and is computationally intractable for large $n$.

We show that by exploiting the combinatorial structure of the observed survey answers via a prefix-tree representation, $\Omega$ contains at most $2(m+1)n$ patterns, enabling complete polynomial-time enumeration.



\begin{lemma}\label{lem:patterns}
$\vert \Omega \vert \leq 2(m+1)n$.
\end{lemma}

\begin{proof}
First, we construct a relaxation $\mathcal{T}$ of $\Omega$. 
For $\y\in X\setminus Q$ and $i\in[n]$, let
\begin{equation*}
d(i,\y):=\min \bigl\{j\in[m] \bigm| y_j \neq q_j^i \bigr\}
\end{equation*}
be the first coordinate at which $\y$ and $\q^i$ differ, and define 
\begin{equation}\label{eq:tau}
\tau^i(\y)
:=
\operatorname{sign}\Bigl(
a^i_{d(i,\y)}
\bigl(y_{d(i,\y)}-q^i_{d(i,\y)}\bigr)
\Bigr).
\end{equation}
This indicator is well-defined because $a_j^i\neq0$ for every $i,j$ and
$y_{d(i,\y)}\neq q^i_{d(i,\y)}$. Now, let $\tau(\y):=(\tau^i(\y))_{i\in[n]}$, and define $\mathcal{T}:= \bigl\{\tau(\y) \bigm| \y \in X\setminus Q\bigr\}$.
We show $\Omega\subseteq\mathcal{T}$ by considering a point $\y$ in $X \setminus Q$ with $\y\in\bigcap_{i\in[n]}C^i$. Let $i \in [n]$.  If $\sigma^i(\y)=+$ then $a_j^i(y_j-q_j^i)\geq0$ for every $j$. At the first differing coordinate $d(i,\y)$, the inequality is strict, so
\begin{equation*}
\sigma^i(\y)= +=\operatorname{sign}\Bigl(
a^i_{d(i,\y)}
\bigl(y_{d(i,\y)}-q^i_{d(i,\y)}\bigr)
\Bigr)=\tau^i(\y).
\end{equation*}
The case $\sigma^i(\y)=-$ is analogous. Hence $\sigma(\y)=\tau(\y)\in\mathcal{T}$, proving $\Omega\subseteq\mathcal{T}$.

To count the orientations in $\mathcal{T}$, we organize the set of survey responses $Q$ by their coordinate prefixes into a data structure termed the \emph{answer tree} (or prefix tree). A coordinate prefix of length $k$ is given by $(q_1^i, \ldots, q_k^i)$. Each node in the tree corresponds to a distinct prefix and stores the subset of observations sharing that prefix. The root node (depth $0$) corresponds to the empty prefix and contains all observations $Q$.   For each non-leaf node $u$ at depth $k \ge 0$, its child nodes correspond to the distinct values of the $(k+1)$-th coordinate among observations in $u$. Each observed answer $\q^i \in Q$ thus traces a unique path from the root to a leaf at depth $m$. Figure~\ref{fig:prefix-tree-example} illustrates the answer tree for a survey with $3$ questions and $4$ rounds.

\begin{figure}[htbp]
  \centering
  \resizebox{\linewidth}{!}{%
    \begin{tikzpicture}[
      every node/.style={
        draw,
        rectangle,
        align=center,
        inner sep=3pt,
        font=\footnotesize
      },
      edge from parent/.style={
        draw,
        -{Latex[length=1.7mm]}
      },
      level 1/.style={
        sibling distance=58mm,
        level distance=16mm
      },
      level 2/.style={
        sibling distance=29mm,
        level distance=16mm
      },
      level 3/.style={
        level distance=16mm
      }
    ]
      \node {prefix: $()$\\ $\{\q^1,\q^2,\q^3,\q^4\}$}
        child {
          node {prefix: $(1)$\\ $\{\q^1,\q^2,\q^4\}$}
          child {
            node {prefix: $(1,2)$\\ $\{\q^1\}$}
            child {
              node {prefix: $(1,2,1)$\\ $\{\q^1\}$}
            }
          }
          child {
            node {prefix: $(1,5)$\\ $\{\q^2\}$}
            child {
              node {prefix: $(1,5,0)$\\ $\{\q^2\}$}
            }
          }
          child {
            node {prefix: $(1,7)$\\ $\{\q^4\}$}
            child {
              node {prefix: $(1,7,2)$\\ $\{\q^4\}$}
            }
          }
        }
        child {
          node {prefix: $(3)$\\ $\{\q^3\}$}
          child {
            node {prefix: $(3,1)$\\ $\{\q^3\}$}
            child {
              node {prefix: $(3,1,2)$\\ $\{\q^3\}$}
            }
          }
        };

      \node[
        draw=none,
        anchor=east,
        inner sep=0pt,
        font=\small
      ] at (-80mm,-24mm) {
        $\begin{aligned}
          \q^1&=(1,2,1),\\
          \q^2&=(1,5,0),\\
          \q^3&=(3,1,2),\\
          \q^4&=(1,7,2).
        \end{aligned}$
      };
    \end{tikzpicture}%
  }
  \caption{An answer tree for $m=3$ questions and $n=4$ rounds. Each node represents a coordinate prefix and contains the answers sharing it; each survey answer corresponds to a unique root-to-leaf path.}
  \label{fig:prefix-tree-example}
\end{figure}
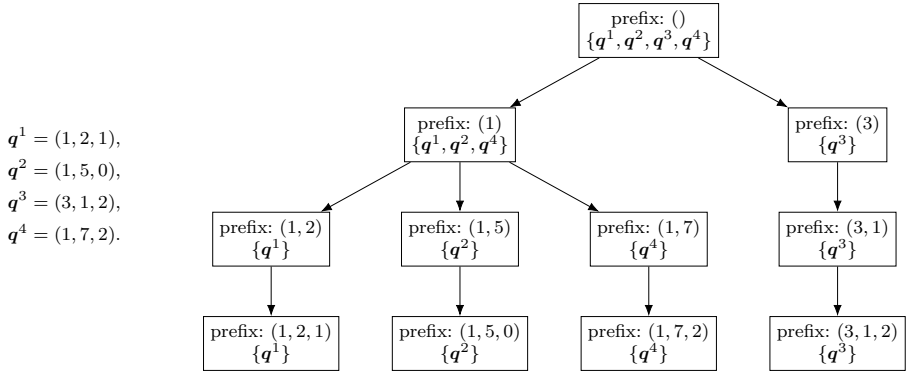

Next, we introduce the notion of \emph{configuration pair} to encode the orientation patterns of $\mathcal{T}$ into the tree. Consider a non-leaf node $u$ at depth $k$. Let $\pi(u)$ denote its prefix and let $N(u)$ denote the index set of observations contained in $u$. Let the distinct $(k+1)^{\text{st}}$ coordinates of the children of $u$ be $v_1<\cdots<v_{d(u)}$, where $d(u)$ denotes the number of children of $u$ and these values are contained in $[0,M_{k+1}]$. The complement $[0,M_{k+1}]\setminus\{v_1,\ldots,v_{d(u)}\}$
is partitioned into at most $d(u)+1$ intervals. A \emph{configuration pair} is a pair $(u,I)$, where $u$ is a non-leaf node and $I$ is one of these intervals. Intuitively, $(u,I)$ records a set of peaks whose first $k$ coordinates coincide with the prefix of $u$ and whose $(k+1)^{\text{st}}$ coordinate lies in $I$. Formally, define
\begin{equation}\label{eq:confi_region}
X(u,I):=\bigl\{\y \in X  \bigm| (y_1,y_2, \ldots, y_k) = \pi(u),\ y_{k+1} \in I\bigr\}.    
\end{equation}

Indeed, we can show that each orientation pattern of $\mathcal{T}$ can be represented by at least one configuration pair. For a configuration pair $(u,I)$, we show that $\tau(\y)$ is constant on  the identified region $X(u,I)$, generating   an orientation pattern of $\mathcal{T}$. In particular, for two distinct peaks $\y$ and $\y'$ in $X(u,I)$, we show that for every $ i \in [n]$, $\tau^i(\y) = \tau^i(\y')$. First, assume that $i\notin N(u)$. Then $\q^i$ differs from the prefix $\pi(u)$ in at least one of the first $k$ coordinates. Since $\y$ and $\y'$ both have prefix $\pi(u)$, there exists some $r\le k$ such that $y_r=y'_r=\pi_r(u)\neq q_r^i$, and, thus, $d(i,\y) = d(i,\y') \leq k$. This implies $\tau^i(\y) = \tau^i(\y')$. Now, consider $i \in N(u)$. Then, $y_r = y'_r = q^i_r$  for $r \leq k$, while both $y_{k+1}$ and $y'_{k+1}$ differ from all child coordinate values $v_1, \ldots,v_{d(u)}$. Thus, $d(i,\y)=d(i,\y') = k+1$. This, together with $y_{k+1}, y'_{k+1} \in I$, shows $\tau^i(\y) = \tau^i(\y')$.  Conversely, consider any $\tau(\y)\in\mathcal{T}$. Since $\y\neq \q^i$ for every $i\in[n]$, $\y$ cannot match a full prefix at depth $m$. Tracing $\y$ down the tree, let $u$ be the unique deepest node, at some depth $k\leq m-1$, whose prefix matches the first $k$ coordinates of $\y$. By maximality, $y_{k+1}\notin\{v_1,\ldots,v_{d(u)}\}$, so $y_{k+1}$ lies in one of the intervals $I$ associated with $u$. Hence the configuration pair $(u,I)$ generates $\tau(\y)$. 

Now, we are ready to finish the counting. Let $T$ be the number of configuration pairs and let $N_{\mathrm{tree}}$ be the number of nodes in the prefix tree. It turns out that 
\begin{equation*}
\begin{aligned}
\vert \Omega \vert& \leq \vert \mathcal{T} \vert \leq T \leq \sum_{u \text{ is non-leaf node}}\bigl(d(u)+1\bigr)\\ 
&\leq(N_{\mathrm{tree}}-1)+N_{\mathrm{tree}}<2N_{\mathrm{tree}}\leq2(m+1)n,  
\end{aligned}
\end{equation*}
Here, $\Omega\subseteq\mathcal{T}$ gives the first inequality. The second inequality follows since each pattern of $\mathcal{T}$ is generated by at least one configuration pair. The third inequality holds since each non-leaf node is associated with $d(u)$ children and thus with at most $d(u) + 1$ intervals. $\sum_u d(u)=N_{\mathrm{tree}}-1$ gives the fourth inequality. The last inequality holds since the tree has depths $0,\ldots,m$, with at most $n$ nodes at each depth. \qed
\end{proof}
The proof of Lemma~\ref{lem:patterns} naturally yields an efficient \emph{enumeration procedure} for listing all elements of $\Omega$:
\begin{enumerate}
    \item Construct the answer tree from $Q$ using trie insertion in $\mathcal{O}(mn \log n)$ time \citep{fredkin1960trie, knuth1998art};
    \item For each pair $(u, I)$, pick an arbitrary sample point $\y \in X(u,I)$ defined as in \eqref{eq:confi_region}, and evaluate the induced pattern $\sigma = \tau(\y)$ via \eqref{eq:tau};
    \item Filter candidate patterns using the membership test \eqref{eq:mem-test}: evaluate the axis-aligned box $\mathcal{R}_\sigma$, and retain $\sigma$ in $\Omega$ if and only if $\mathcal{R}_\sigma \setminus Q \neq \varnothing$.
\end{enumerate}
This membership test can be verified in polynomial-time:  $\mathcal{R}_\sigma$ is an axis-aligned box and can be obtained by coordinate-wise computation. An empty box fails the test, while a nonempty box with a positive-length side necessarily contains a point outside the finite set $Q$; if the box is a singleton, we simply check whether its unique point belongs to $Q$. 

Since every orientation in $\Omega$ is generated by at least one configuration pair, after screening all such pairs, the procedure enumerates exactly the elements of $\Omega$. Since the number of configuration pairs is at most $2(m+1)n$ by Lemma~\ref{lem:patterns}, all membership tests can be performed in polynomial time. Thus, the overall running time of the enumeration procedure is polynomial in $m$ and $n$.

\begin{algorithm}[!htb]
\caption{Polynomial-time test of peak-oriented rationalizability}
\label{alg:construct-omega}
\begin{algorithmic}[1]
\REQUIRE Survey dataset $D$ 
\STATE $\Omega \gets$ enumeration procedure applied to $D$ \label{alg:construct-omega-1}
\FOR{$\sigma \in \Omega$}\label{alg:construct-omega-2}
\IF{$\{(\q^i,\sigma^i \a^i)\}_{i\in[n]}$ satisfies~\eqref{eq:GARP}}\label{alg:construct-omega-3}
\STATE Choose $\y^* \in \mathcal{R}_\sigma \setminus Q $\label{alg:construct-omega-4}
    \RETURN  $D$ is peak-oriented rationalizable with peak $\y^*$ \label{alg:construct-omega-5}
\ENDIF \label{alg:construct-omega-6}
\ENDFOR \label{alg:construct-omega-7}

\FOR{each $\y\in Q$} \label{alg:construct-omega-8}
    \STATE $I\gets\{i\in[n]\mid \q^i\neq \y\}$ \label{alg:construct-omega-9}
    \IF{$\y\in C^i$ for every $i\in I$} \label{alg:construct-omega-10}
        \IF{$\{(\q^i,\sigma^i(\y)\a^i)\}_{i\in I}$ satisfies~\eqref{eq:GARP}} \label{alg:construct-omega-11}
    \RETURN  $D$ is peak-oriented rationalizable with peak $\y$  \label{alg:construct-omega-12}
        \ENDIF \label{alg:construct-omega-13}
    \ENDIF \label{alg:construct-omega-14}
\ENDFOR \label{alg:construct-omega-15}
\RETURN $D$ is not peak-oriented rationalizable \label{alg:construct-omega-16}
\end{algorithmic}
\end{algorithm}

Algorithm~\ref{alg:construct-omega} presents the complete rationalizability test. It exhaustively partitions candidate peaks into two cases: peaks lying strictly outside $Q$ (Lines~\ref{alg:construct-omega-1}--\ref{alg:construct-omega-7}) and peaks coinciding with an observed answer in $Q$ (Lines~\ref{alg:construct-omega-8}--\ref{alg:construct-omega-15}).  In the first case, the algorithm enumerates $\Omega$ and tests~\eqref{eq:GARP} for the corresponding reoriented dataset for each $\sigma\in\Omega$. By Theorem~\ref{them:peak-garp}, a successful test yields a valid peak in $\mathcal{R}_\sigma\setminus Q$. In the second case, the algorithm considers every $\y\in Q$, checks the corresponding cone-membership conditions, and applies the GARP test to the observations with $\q^i\neq\y$. These two cases exhaust all possible locations of the peak. Thus, if both cases fail, $D$ is not peak-oriented rationalizable. By Lemma~\ref{lem:patterns}, the enumeration procedure generates only polynomially many orientation patterns in polynomial time. Since each GARP and cone-membership check is also polynomial-time, Algorithm~\ref{alg:construct-omega} runs in polynomial time. We thus obtain the following theorem.

\begin{theorem}\label{thm:polynomial}
Given a survey dataset $D$ with $m$ survey questions and $n$ rounds, Algorithm~\ref{alg:construct-omega} determines whether $D$ is peak-oriented rationalizable in time polynomial in $m$ and $n$.
\end{theorem}

In Algorithm~\ref{alg:construct-omega}, once a rationalizing peak $\y^*$ is found, a utility function can be constructed from an oriented Afriat system~\citep[Proposition 2]{seror2026concaverationalizationidealpoint}. More specifically,  find $U^0$, $U^i$, and $\eta^i>0$ for $i \in I(\y^*)$ satisfying
\begin{equation}\label{eq:afriat}
\begin{aligned}
& U^k \leq U^i+  \eta^i \bigl\langle \sigma^i(\y^*) \a^i, \q^k-\q^i \bigr\rangle \quad && \text{ for } i,k\in I(\y^*),\\
& U^i \leq U^0\leq U^i+ \eta^i\bigl\langle \sigma^i(\y^*) \a^i, \y^*-\q^i \bigr\rangle && \text{ for } i\in I(\y^*).
\end{aligned}
\end{equation}
For any such solution, a utility function $u: X \to \mathbb{R}$ that yields a peak-oriented rationalization of $D$ is given by
\begin{equation}\label{eq:utility}
u(\x):=\min\left\{U^0,\ \min_{i\in I(\y^*)}\Bigl[U^i+\eta^i \bigl\langle \sigma^i(\y^*) \a^i, \x- \q^i \bigr\rangle \Bigr]\right\}.
\end{equation}

\begin{corollary}\label{cor:utility-construction}
If a survey dataset $D$ is peak-oriented rationalizable, a rationalizing utility function $u(\x)$ can be obtained in time polynomial in $m$ and $n$.
\end{corollary}

\begin{proof}
Since $D$ is peak-oriented rationalizable, apply Algorithm~\ref{alg:construct-omega} to $D$ and it returns a rationalizing peak $\y^*$. Let $I:=I(\y^*)$. If $I=\varnothing$, the constant utility function $u(\x) =  0$ suffices. Otherwise, the system~\eqref{eq:afriat} admits a solution with $\eta^i>0$ for every $i\in I$.

We show that the system \eqref{eq:afriat} is feasible with $\eta^i>0$ for every $i\in I$ if and only if it is feasible with $\eta^i\geq1$ for every $i\in I$. For the ``only if'' direction, take a feasible solution with $\eta^i>0$ for every $i\in I$ and let $\delta:=\min_{i\in I}\eta^i>0$. Define
\begin{equation*}
\widehat U^0:=\frac{U^0}{\delta},\quad
\widehat U^i:=\frac{U^i}{\delta},\quad
\widehat\eta^i:=\frac{\eta^i}{\delta}
\quad\text{for }i\in I.
\end{equation*}
Since all inequalities in~\eqref{eq:afriat} are homogeneous in $U^0$, $U^i$, and $\eta^i$ for every $i\in I$, dividing both sides by $\delta$ shows that the rescaled variables also satisfy~\eqref{eq:afriat}. Moreover, the definition of $\delta$ gives $\widehat\eta^i\geq1$ for every $i\in I$. Therefore, the system with $\eta^i\geq1$ for every $i\in I$ is feasible. For the reverse direction, since $\eta^i\geq1$ implies $\eta^i>0$, any feasible solution under the former constraints is also feasible under the latter.

Consequently, finding a valid parameter set, $U^0,U^i$ and  $\eta^i$, reduces to solving a standard linear program over system~\eqref{eq:afriat} augmented with the linear constraints $\eta^i \ge 1$ for $i \in I$. Since this linear program has polynomial size in $m$ and $n$, a feasible solution can be computed in polynomial time via standard linear programming algorithms, \textit{e.g.}, the interior-point algorithm~\cite{karmarkar1984new}. Substituting these parameters into~\eqref{eq:utility} yields the desired utility function $u(\x)$.\qed

\end{proof}




\section{The Peak-Oriented Houtman--Maks Index}\label{hardness}
The previous section shows that peak-oriented rationalizability can be tested in polynomial time, and that a rationalizing utility function can be constructed whenever the test succeeds. In survey analysis, however, a respondent's answers may fail exact rationalizability because only a small number of rounds are inconsistent with a common peak. A binary True/False test does not distinguish such cases from datasets with widespread violations. 

Following the Houtman--Maks approach to measuring goodness of fit~\citep{houtman1985determining}, \citet{seror2026concaverationalizationidealpoint} therefore considers a consistency index for peak-oriented rationalization, measuring the largest fraction of survey observations that can be rationalized by a common peak, that is $\HM_{\mathrm{Peak}}(D)/\vert D \vert$, where 
\begin{equation*}
\mathrm{HM}_{\mathrm{Peak}}(D):=
\max\bigl\{\vert S \vert  \bigm| 
S\subseteq D \text{ is peak-oriented rationalizable}\bigr\}.
\end{equation*}
The index therefore provides a quantitative measure of how closely a respondent's survey data conform to peak-oriented rationality, when the full dataset fails the exact rationalization test. 

This raises a natural computational question: does the polynomial-time test for exact rationalizability extend to computing this goodness-of-fit measure? The answer is negative.

\begin{theorem}\label{thm:nphard}
Given a survey dataset $D$, computing  $\mathrm{HM}_{\mathrm{Peak}}(D) $ is NP-Hard.
\end{theorem}

\begin{proof}
We prove the result by reducing from the Houtman--Maks problem \citep{houtman1985determining} in consumer theory, defined as follows. Let  $E= \bigl\{(\x^t,\p^t) \bigr\}_{t\in[T]}$ be a set of consumption bundles and  prices vectors, with $T \ge 2$. The Houtman--Maks index of $E$ is defined as $ \HM(E)/ \vert E \vert$, where
\[
\HM(E):= \max\bigl\{ \vert H \vert \bigm| H \subseteq E \text{ satisfies GARP} \bigr\}.
\]
Given  a consumption dataset $E$  and a constant $L$, deciding whether $\HM(E) \geq L$ is NP-complete~\citep[Theorem 5.1]{smeulders2014goodness}. 

Given a consumption dataset $E$, we construct a survey dataset $D$ as follows:
\begin{itemize}
	\item The number of questions  is $T+1$, and the set of all possible answers is  $X = [0,1]^{T+1}$.
	\item The number of rounds is $T + (T+1)$. For each round $i$ of the first $T$ rounds, the answer $\q^i = (e_i,0)$, the budget normal vector $\a^i = (\alpha^i_1, \ldots, \alpha^i_T,1)$, where $e_i$ is the $i$-th unit vector in $\mathbb{R}^T$ and 
\[
\alpha_j^i:=1+\tfrac12\operatorname{sign}\! \left( \langle \p^i ,\x^j-\x^i \rangle \right) \in  \left\{\tfrac12,1,\tfrac32\right\}.
\]
For each round $i \in \{T+1, \ldots, 2T+1\}$, the answer $\q^i = \mathbf{1}_{T+1}$ and  $\a^i = (\mathbf{1}_{T},-1)$.	
\item For $i \in  [2T+1]$, let $B^i:= \bigl\{\x \in X \bigm| \langle \a^i, \x \rangle = \langle \a^i,\q^i\rangle   \bigr\}$. Then, the survey dataset $D := \{(\q^i, B^i)\}_{i \in [2T+1]}$.
\end{itemize}
This construction has three useful properties. First, $D$ satisfies assumptions that all coordinates of $\a^i$ are nonzero, $B^i$ intersects the interior of $X$, and $\q^i\in B^i$. Second, 
we obtain the following sign equality
\begin{equation}\label{eq:sign-iso}
\operatorname{sign}\! \left( \langle \a^i,  \q^k- \q^i \rangle \right) =\operatorname{sign}\! \left( \langle \p^i , \x^k- \x^i \rangle \right) \quad \text{ for } i,k\in[T].
\tag{$\ast$}
\end{equation}
This holds because  $\langle \a^i, \q^k-\q^i \rangle = \alpha^i_k - \alpha^i_i = \frac{1}{2}\operatorname{sign}\bigl(\langle \p^i,\x^k - \x^i \rangle \bigr)$. Last, for each \(i\in[T]\), since $\q^i=(e_i,0)$ and every component of \(\a^i\) is positive, we have
\[
\begin{aligned}
O_+^i &=\bigl\{\y\in[0,1]^{T+1}\bigm| y_i=1\bigr\},\\
O_-^i
&=\bigl\{\y\in[0,1]^{T+1} \bigm|
y_j=0\ \text{for all }j\in[T]\setminus\{i\},
\ y_{T+1}=0\bigr\}.
\end{aligned}
\]
For each $i\in\{T+1,\ldots,2T+1\}$, since $\q^i=\mathbf 1_{T+1}$ and the first $T$ components of $\a^i$ are positive while its last component is negative, we have
\[
\begin{aligned}
O_+^i
&=\bigl\{\y\in[0,1]^{T+1} \bigm|
y_j=1\ \text{for all }j\in[T]\bigr\},\\
O_-^i
&=\bigl\{\y\in[0,1]^{T+1}\bigm| y_{T+1}=1\bigr\}.
\end{aligned}
\]

We now show that $\HM_{\mathrm{Peak}}(D)=\HM(E)+T+1$.  It follows that an algorithm computing
$\HM_{\mathrm{Peak}}(D)$ would allow us to compute
$\HM(E)$ by subtracting $T+1$. Equivalently, for any
integer $L$, $\HM(E)\ge L$ if and only if $\HM_{\mathrm{Peak}}(D)\ge L+T+1$.  The construction is computable in polynomial time, so a polynomial-time algorithm for computing \(\HM_{\mathrm{Peak}}(D)\) would solve the NP-complete Houtman--Maks decision problem in polynomial time. Hence computing \(\HM_{\mathrm{Peak}}(D)\) is NP-Hard.

\paragraph{Lower bound.}
We first prove that $\HM_{\mathrm{Peak}}(D)\ge\HM(E)+T+1$.
Let $I\subseteq[T]$ be an index set such that
$\{(\x^t,\p^t)\}_{t\in I}$ satisfies GARP and
$|I|=\HM(E)$. We show that the survey data subset
\[
D_I:=
\bigl\{(\q^i,\a^i) \bigm| i\in I\cup\{T+1,\ldots,2T+1\}\bigr\}
\]
is peak-oriented rationalizable with peak $\y^*=\mathbf 1_{T+1}$.

By construction, $I(\y^*)=[T]$, and $\y^*\in O_+^i$ for every $i\in[T]$. Thus, condition (a) of Theorem~\ref{them:peak-garp} is satisfied. Moreover, for all $i,k\in[T]$,
\[
\operatorname{sign}\bigl(
\langle\sigma^i(\y^*)\a^i,\q^k-\q^i\rangle\bigr)=
\operatorname{sign}\bigl(
\langle\a^i,\q^k-\q^i\rangle\bigr) =
\operatorname{sign}\bigl(
\langle\p^i,\x^k-\x^i\rangle\bigr),
\]
where the first equality follows from
$\sigma^i(\y^*)=+$, and the second follows from
\eqref{eq:sign-iso}. Hence, since
$\{(\x^t,\p^t)\}_{t\in I}$ satisfies GARP, the corresponding oriented survey dataset
\[
\bigl\{(\sigma^i(\y^*)\a^i,\q^i)\bigr\}_{i\in I}
\]
also satisfies GARP. The remaining $T+1$ observations are located at the peak $\y^*$ and therefore do not enter the oriented GARP test. Thus, condition (b) of Theorem~\ref{them:peak-garp} is also satisfied, so $D_I$ is peak-oriented rationalizable. By the definition of $\HM_{\mathrm{Peak}}(D)$, it follows that
\[
\HM_{\mathrm{Peak}}(D)
\ge |D_I| = |I|+T+1
=\HM(E)+T+1.
\]

\paragraph{Upper bound.}
We next prove that $\HM_{\mathrm{Peak}}(D)\le\HM(E)+T+1$. Let \(I\subseteq[2T+1]\) such that $\bigl\{(\q^i,\a^i)\bigr\}_{i\in I}$
is peak-oriented rationalizable with peak $\y^*$, and
$|I|=\HM_{\mathrm{Peak}}(D)$. Define a partition of $I$
 \[
I':=I\cap[T] \quad \text{ and } \quad
I'':=I\cap\{T+1,\ldots,2T+1\}.
\]
We first show that $\y^*\in C^{T+1}$. Suppose, to the contrary, that $\y^*\notin C^{T+1}$. Then none of the observations indexed by $\{T+1,\ldots,2T+1\}$ can satisfy condition (a) of Theorem~\ref{them:peak-garp}. Consequently, $I''=\varnothing$, and hence $|I|\le T$. Since $|I|=\HM_{\mathrm{Peak}}(D)$, this contradicts the lower bound $\HM_{\mathrm{Peak}}(D)\ge\HM(E)+T+1\ge T+1$. Therefore, $\y^*\in C^{T+1}$.

For every $i\in  I' \subseteq [T]$, condition (a) implies that $\y^*\in C^i = O^i_+ \cup O^i_-$. Since $T \ge 2$, we have $C^{T+1}\cap O_-^i=\varnothing$. It follows that $\y^*\in O_+^i \setminus \{\q^i\}$, and thus $\sigma^i(\y^*)=+$ for $i \in I'$. For every $i,k\in I'$, the equality~\eqref{eq:sign-iso}  and the positivity of $\sigma^i(\y^*)$ shows that   
\[
\begin{aligned}
\operatorname{sign}\bigl(
\langle\p^i,\x^k-\x^i\rangle\bigr)
&=\operatorname{sign}\bigl(
\langle\a^i,\q^k-\q^i\rangle\bigr)\\
&=\operatorname{sign}\bigl(
\langle\sigma^i(\y^*)\a^i,\q^k-\q^i\rangle\bigr),
\end{aligned}
\]
 Because the oriented survey data $\bigl\{(\sigma^i(\y^*)\a^i, \q^i) \bigr\}_{i \in I'}$ satisfies GARP, this equality implies that
$\{(\x^t,\p^t)\}_{t\in I'}$ satisfies GARP. Therefore,
$|I'|\le\HM(E)$. Since $|I''|\le T+1$, we obtain
\[
\HM_{\mathrm{Peak}}(D) =|I|=|I'|+|I''| \leq\HM(E)+T+1.
\] 
This completes the proof. \qed
\end{proof}

\section{Conclusion}
In this paper, we establish the exact computational boundaries for peak-oriented revealed preference analysis of survey data. We resolve a Varian-type question by presenting a polynomial-time algorithm to test peak-oriented rationalizability and construct a rationalizing utility function. The cornerstone of our positive result is a prefix-tree construction that leverages the combinatorial structure of survey answers to reduce an otherwise exponential orientation patterns to a polynomial number of relevant ones. Conversely, we show that measuring the extent of consistency via the peak-oriented Houtman--Maks index is NP-hard. Taken together, our results demonstrate that while exact preference elicitation under peak-orientation is computationally tractable, goodness-of-fit analysis for imperfect survey data is inherently intractable unless P=NP.

Motivated by the success of integer programming approach for address NP-Hard problems in revealed preference theory~\citep{cherchye2015revealed,demuynck2023computing}, a promising direction for future research is to develop an exact integer programming formulation for computing the peak-oriented Houtman–Maks index. 



\section*{Statements and Declarations}

\textbf{Financial interest} The authors did not receive support from any organization for the submitted work.

\noindent \textbf{Competing Interests} The authors have no competing interests, conflict of interest, or non-financial interests to declare that are relevant to the content of the article.

\bibliographystyle{spbasic}
\bibliography{references}

\end{document}